\documentclass[12pt]{amsart}
\usepackage{geometry}
\usepackage{amssymb}
\usepackage{amsmath}
\usepackage{amsfonts}
\usepackage{amsthm}
\usepackage{mathrsfs}
\usepackage{graphicx}
\usepackage[table]{xcolor}
\usepackage{hyperref}
\usepackage[all]{xy}
\usepackage[active]{srcltx}
\usepackage{url}
\usepackage[active]{srcltx}
\usepackage{tabularx}

\DeclareMathOperator{\SL}{SL}
\DeclareMathOperator{\Pic}{Pic}

\DeclareMathOperator{\Tr}{Tr}

\DeclareMathOperator{\Gal}{Gal}
\DeclareMathOperator{\Jac}{Jac}

\DeclareMathOperator{\End}{End}

\DeclareMathOperator{\ord}{ord}
\DeclareMathOperator{\sign}{sign}
\DeclareMathOperator{\cond}{cond}

\DeclareMathOperator{\Disc}{disc}
\def\id#1{{\mathfrak{#1}}}      
\def \QQ{\mathbb{Q}}
\def \AA{\mathbb{A}}

\def \ZZ{\mathbb{Z}}

\def\On{\mathcal{O}}
\def\Om{\mathscr{O}}

\def \CC{\mathbb{C}}
\def\<#1>{{\left\langle{#1}\right\rangle}}

\let\kro\dkro

\newcommand{\smtx}[4]{\left(\begin{smallmatrix}#1&#2\\#3&#4\end{smallmatrix}\right)}
\newcommand{\cO}{\mathcal{O}}
\newcommand{\cH}{\mathcal{H}}
\newcommand{\tns}{\otimes}
\newcommand{\Div}{\operatorname{Div}}

\theoremstyle{plain}
\newtheorem{thm}{Theorem}[section]
\newtheorem{prop}[thm]{Proposition}
\newtheorem{coro}[thm]{Corollary}
\newtheorem{lemma}[thm]{Lemma}

\theoremstyle{remark}
\newtheorem{rem}[thm]{Remark}

\newtheorem*{rem*}{Remark}
\theoremstyle{definition}
\newtheorem{obs/res}[thm]{Observation}

\begin{document}
\title[Ramified Heegner points]{On Heegner Points for primes of additive reduction ramifying in the base field}

\author{Daniel Kohen}
\address{Departamento de Matem\'atica, Facultad de Ciencias Exactas y Naturales, Universidad de Buenos Aires and IMAS, CONICET, Argentina}
\email{dkohen@dm.uba.ar}
\thanks{DK was partially supported by a CONICET doctoral fellowship}

\author{Ariel Pacetti}
\address{Departamento de Matem\'atica, Facultad de Ciencias Exactas y Naturales, Universidad de Buenos Aires and IMAS, CONICET, Argentina}
\email{apacetti@dm.uba.ar}
\thanks{AP was partially supported by CONICET PIP 2010-2012 11220090100801, ANPCyT PICT-2013-0294 and  UBACyT 2014-2017-20020130100143BA}

\address{University of Warwick, Coventry, UK}
\email{m.masdeu@warwick.ac.uk}
\thanks{MM was supported by EU H2020-MSCA-IF-655662}

\dedicatory{with an Appendix by Marc Masdeu}

\begin{abstract}
  Let $E$ be a rational elliptic curve, and $K$ be an imaginary
  quadratic field. In this article we give a method to construct
  Heegner points when $E$ has a prime bigger than $3$ of additive
  reduction ramifying in the field $K$. The ideas apply to more
  general contexts, like constructing Darmon points attached to real
  quadratic fields which is presented in the appendix.
\end{abstract}				
\maketitle

\section*{Introduction}

Heegner points play a crucial role in our nowadays understanding of
the Birch and Swinnerton-Dyer conjecture, and are the only instances
where non-torsion points can be constructed in a systematic way for
elliptic curves over totally real fields (assuming some still unproven
modularity hypotheses). Although Heegner points were heavily studied
for many years, most applications work under the so
called ``Heegner hypothesis'' which gives a sufficient condition
for an explicit construction to hold. In general, if $E$ is an
elliptic curve over a number field $F$ and $K/F$ is any quadratic
extension, the following should be true.

\medskip

\noindent{\bf Conjecture:} If $\sign(E,K)=-1$, then there is a non-trivial Heegner system attached to $(E,K)$.

\medskip

This is stated as Conjecture 3.16 in \cite{Dar04}. When $F=\QQ$, $E$
is an elliptic curve of square-free conductor $N$ and $K$ is an
imaginary quadratic field whose discriminant is prime to $N$, the
conjecture is proven in Darmon's book (\cite{Dar04}) using both the
modular curve $X_0(N)$ and other Shimura curves. The hypotheses on $N$
and $K$ were relaxed by Zhang in \cite{ZH}, who proved the
conjecture under the assumption that if a prime $p$ ramifies in $K$ then
$p^2 \nmid N$.

When the curve is not semistable at some prime $p$ the situation is
quite more delicate. An interesting phenomenon is that in this
situation, the local root number at $p$ has no relation with the
factorization of $p$ in $K$. Still the problem has a positive answer
in full generality, due to the recent results of \cite{yuan2013gross}, where
instead of working with the classical group $\Gamma_0(N)$, they deal
with more general arithmetic groups. The purpose of this article is to give
``explicit'' constructions of Heegner points for pairs $(E,K)$ as
above. Here by explicit we mean that we can compute numerically the theoretical
points in the corresponding ring class field, which restricts
us to working only with unramified quaternion algebras (since the
modular parametrization is hard to compute for Shimura curves). For
computational simplicity we will also restrict the base field to
the field of rational numbers.

Let $\chi:K^\times \backslash K^\times_{\AA}\to \CC^\times$ be a
finite order anticyclotomic Hecke character, and $\eta$ be the character
corresponding to the quadratic extension $K/\QQ$. In order to
construct a Heegner point attached to $\chi$ in a matrix algebra, for
each prime number $p$ the following condition must hold
\[
\epsilon(\pi_p,\chi_p)=\chi_p(-1) \eta_p(-1),
\]
where $\pi$ is the automorphic representation attached to $E$, and
$\epsilon(\pi_p,\chi_p)$ is the local root number of $L(s,\pi,\chi)$
(see \cite[Section 1.3.2]{yuan2013gross}). If we impose the extra condition
$\gcd(\cond(\chi),N\cond(\eta))=1$, then at
primes dividing the conductor of $E/K$ the equation becomes
\[
\varepsilon_p(E/K)=\eta_p(-1),
\]
where $\varepsilon_p(E/K)$ is the local root number at $p$ of the base change
of $E$ to $K$ (it is equal to $\varepsilon_p(E) \varepsilon_p(E\otimes
\eta)$). This root number is easy to compute if $p \neq 2,3$ (see \cite{Pacetti}):
\begin{itemize}
\item If $p$ is unramified in $K$, then $\eta_p(-1)=1$ and
\[
\varepsilon_p(E/K)=\begin{cases}
1 & \text{ if } v_p(N)=0,\\
\kro{p}{\Disc(K)} & \text{ if } v_p(N)=1,\\
1 & \text{ if } v_p(N)=2,
\end{cases}
\]
where $v_p(N)$ denotes the valuation of $N$ at $p$.
\item If $p$ is ramified in $K$ then $\eta_p(-1)=\kro{-1}{p}$ and
\[
\varepsilon_p(E/K)= \kro{-1}{p} \cdot
\begin{cases}
1 & \text{ if } v_p(N)=0 ,\\
\varepsilon_p(E) & \text{ if } v_p(N)=1 ,\\
\varepsilon_p(E_p) & \text{ if } v_p(N_{E_p})=1,\\
1 & \text{ if }E \text{ is P.S.} ,\\
-1 & \text{ if } E \text{ is S.C.},\\
\end{cases}
\]
where $E_p$ denotes the quadratic twist of $E$ by the unique quadratic
extension of $\QQ$ unramified outside $p$; $E$ is P.S. if the attached
automorphic representation is a ramified principal series (which is
equivalent to the condition that $E$ acquires good reduction over an
abelian extension of $\QQ_p$) and $E$ is S.C. if the attached
automorphic representation is supercuspidal at $p$ (which is
equivalent to the condition that $E$ acquires good reduction over a
non-abelian extension).  
\end{itemize}

Let $E/\QQ$ be an elliptic curve. We call it Steinberg at a prime $p$
if $E$ has multiplicative reduction at $p$ (and denote it by St.). In
Table~\ref{table:signs} we summarize the above equations for
$p \neq 2,3$, where the sign corresponds to the product
$\varepsilon_p(E/K) \eta_p(-1)$.
\begin{table}[h]

\begin{tabular}{l|r|r|r}
& $p$ is inert & $p$ splits & $p$ ramifies\\
\hline
St & \cellcolor[gray]{0.9} $-1$ & \cellcolor[gray]{0.9} $1$ & \cellcolor[gray]{0.9}$\varepsilon_p(E)$\\
\hline
St $\otimes \chi_p$ &\cellcolor[gray]{0.7} $1$ & \cellcolor[gray]{0.9}$1$ & $\varepsilon_p(E_p)$\\
\hline
P.S. &\cellcolor[gray]{0.7} $1$ &\cellcolor[gray]{0.9} $1$ & $1$\\
\hline
Sc. &\cellcolor[gray]{0.7} $1$ &\cellcolor[gray]{0.9} $1$ & $-1$
\end{tabular}
\caption{Signs Table}
\label{table:signs}
\end{table}

Our goal is to give an explicit construction in all cases where
the local sign of Table~\ref{table:signs} equals $+1$. The cells
colored in light grey correspond to the classical construction, and
the ones colored with dark grey are considered in the article
\cite{Kohen}. In the present article we will consider the following cases:
\begin{itemize}
\item $E$ has additive but potentially multiplicative reduction, and  $\varepsilon_{p}(E_{p})=+1$.
\item $E$ has additive but potentially good reduction over an abelian extension.
\end{itemize}

\begin{rem*}
  The situation for $p=2$ and $p=3$ is more delicate, although most
  cases can be solved with the same ideas. For the rest of this
  article we assume $p>3$.
\end{rem*}
The strategy is to build an abelian variety related to $E/K$ (in general of
dimension greater than $1$) and use a classical Heegner construction
on such variety so that we can transfer the Heegner points back to our
original elliptic curve. To clarify the exposition, we start assuming
that there is only one prime $p$ ramifying in $K$ where our curve has
additive reduction, and every other prime $q$ dividing $N$ is split in
$K$. The geometric object we consider is the following:

\begin{itemize}
\item If $E$ has potentially multiplicative reduction, we consider the
  elliptic curve $E_p$ of conductor $N/p$ which is the quadratic twist
  of $E$ by the unique quadratic character ramified only at $p$.
\item If $E$ has potentially good reduction over an abelian extension, then we consider an
  abelian surface of conductor $N/p$, which is attached to a pair
  $(g,\bar{g})$, where $g$ is the newform of level $N/p$ corresponding
  to a twist of the weight $2$ modular form $E_f$ attached to $E$.
\end{itemize}
In both cases the classical Heegner hypothesis is satisfied
 (eventually for
dimension greater than one), and the resulting abelian varieties are
 isogenous to our curve or to a
product of the curve with itself over some field extension. Such
isogeny is the key to relate the classical construction to the new
cases considered. Each case has a different construction/proof (so
they will be treated separately), but both follow the same idea.  In
all cases considered we will construct points on
$(E(H_{c}) \otimes \CC)^{\chi}$. These points will be non-torsion if
an only if $L'(E/K, \chi ,1) \neq 0$ as expected by the results of
Gross-Zagier  \cite{GZ}  and Zhang \cite{ZH}.

Our construction is interesting on its own, and can be used to
move from a delicate situation to a not so bad one (reducing the
conductor of the curve at the cost of adding a character in some
cases). So, despite we focus on classical modular curves, the methods of this article can be easily applied to a wide variety of contexts, for example
more general Shimura curves.

In recent years, following a breakthrough idea of Darmon there has been a lot of work in the direction of defining and computing $p$-adic Darmon points, which are points defined over 
certain ring class fields of real quadratic extensions using $p$-adic methods. For references to this circle of ideas
the reader can consult \cite{Dar04}, \cite{darmon2001integration}, \cite{bertolini2009rationality}, \cite{MR2289868}. 
 These construction are mostly conjectural (but see \cite{bertolini2009rationality}), and there has been a lot of effort to explicitly compute $p$-adic approximations to these points in order to
 gather numerical evidence supporting these conjectures. The interested reader might consult \cite{darmon2006efficient}, \cite{MR2510743}, \cite{guitart2015elementary},
 \cite{MR3384519}, \cite{MR3418066}.

In order to illustrate the decoupling of our techniques from the algebraic origin of the points, in an appendix by Marc Masdeu it is shown how these can be applied to the computation of $p$-adic Darmon points.

The article is organized as follows: in the first section we treat the
case of a curve having potentially multiplicative reduction, and prove
the main result in such case. In the second section we prove our main
result in the case that we have potentially good reduction over an abelian extension. In the third section, we
explain how to extend the result to general conductors and in the fourth section we finish
the article with some explicit examples in the modular curves setting, including Cartan non-split curves, as in \cite{Kohen}. Lastly, we include the aforementioned appendix.

\bigskip

\noindent {\bf Acknowledgments:} We would like to thank Henri Darmon
for many comments and suggestions regarding the present article and
Marc Masdeu for his great help and contributions to this project.  We
would also like to thank the Special Semester ``Computational aspects
of the Langlands program'' held at ICERM for providing a great
atmosphere for working on this subject. Finally, we would like to thank the referee for the useful remarks.

\section{The potentially multiplicative case}
\label{sec:potmult}
Let $E/\QQ$ be an elliptic curve of conductor $p^2 \cdot m$ where $p$
is an odd prime and $gcd(p,m)=1$. Suppose that $E$ has potentially
multiplicative reduction at the prime $p$. Let $K$ be any imaginary
quadratic field satisfying the Heegner hypothesis at all the primes
dividing $m$ and such that $p$ is ramified in $K$.  Let
$p^{\ast}=\kro{-1}{p}p$ and let $E_p$ be the quadratic twist of $E$ by
$\QQ(\sqrt{p^\ast})$. We have an isomorphism $\phi:E_p \to E$ defined
over $\QQ(\sqrt{p^\ast})$. The elliptic curve $E_p$ has conductor
$p \cdot m$ and $\sign(E,K)=\sign(E_p,K)\varepsilon_p(E_p)$.

Recall that to have explicit constructions, we need to work with a
matrix algebra so we impose the condition $\varepsilon_p(E_p)=1$ (see
Table \ref{table:signs}).  Then, $\sign(E,K)=\sign(E_p,K)=-1$ and the
pair $(E_{p},K)$ satisfies the Heegner condition. Therefore, we can
find Heegner points on $E_{p}$ and map them to $E$ via $\phi$. More
precisely, let $c$ be a positive integer relatively prime to
$N \cdot \Disc(K)$ and let $H_{c}$ be the ring class field associated
to the order of conductor $c$ in the ring of integers of $K$. Let
$\chi : \Gal(H_{c}/K) \rightarrow \CC^{\times}$ be any character and
let $\chi_{p}$ be the quadratic character associated to
$\QQ(\sqrt{p^\ast})$ via class field theory.  Take a Heegner point
$P_{c} \in E_p(H_c)$ and consider the point

\[P^{\chi\chi_{p}}_{c}= \sum_{\sigma \in \Gal(H_{c}/K)} \bar{\chi}\bar{\chi_{p}}(\sigma)
P_{c}^{\sigma} \in (E_p(H_{c})\otimes \CC)^{\chi \chi_{p}}.\]

\begin{thm}
The point $\phi(P^{\chi\chi_{p}}_{c})$ belongs to $(E(H_{c}) \otimes
\CC)^{\chi}$ and it is non-torsion if and only if $L'(E/K,\chi,1)  \neq 0$.
\end{thm}
\begin{proof}
The key point is that since $p\mid \Disc(K)$,
$\QQ(\sqrt{p^\ast}) \subset H_c$ (by genus theory). For $\sigma \in \Gal(\bar{\QQ}/\QQ)$,
we have $\phi^{\sigma}= \chi_{p}(\sigma)\phi$, hence,  
\[
\phi(P^{\chi \chi_{p}}_{c})=\sum_\sigma \bar{\chi}(\sigma) \phi(P_{c})^{\sigma} \in (E(H_{c}) \otimes
\CC)^{\chi}.
\]
Finally note that by the Theorems of Gross-Zagier \cite{GZ} and Zhang \cite{ZH} the point $P^{\chi\chi_{p}}_{c}$ is non-torsion if and only if
$L'(E_{p} /K,\chi \chi_{p},1)= L'(E/K,\chi,1)  \neq 0$. Since $\phi$ is an isomorphism the result follows.

\end{proof}

\section{The potentially good case (over an abelian extension)}
Let $E/\QQ$ be an elliptic curve of conductor $p^2 \cdot m$ where $p$
is an odd prime and $gcd(p,m)=1$. For simplicity assume that $E$ does
not have complex multiplication.  We recall some generalities on
elliptic curves with additive but potentially good reduction over an
abelian extension. Although such results can be stated and explained
using the theory of elliptic curves, we believe that a representation
theoretical approach is more general and clear. Let $f_E$ denote the
weight $2$ newform corresponding to $E$.

Let $W(\QQ_p)$ be the Weil group of $\QQ_p$, and $\omega_1$ be the
unramified quasi-character giving the action of $W(\QQ_p)$ on the
roots of unity. Using the normalization given by Carayol
(\cite{Carayol}), at the prime $p$ the Weil-Deligne representation corresponds to a principal series representation
on the automorphic side and to a representation
\[
\rho_p(f)=\psi \oplus \psi^{-1} \omega_1^{-1},
\]
on the Galois side for some quasi-character
$\psi:W(\QQ_p)^{ab} \to \CC^\times$. Note that since the trace lies in
$\QQ$, $\psi$ satisfies a quadratic relation, hence its image lies in
a quadratic field contained in a cyclotomic extension (since $\psi$ has
finite order). This gives the following possibilities for the order of
inertia of $\psi$: $1$, $2$, $3$, $4$ or $6$.
\begin{itemize}
\item Clearly $\psi$ cannot have order $1$ (since otherwise the
representation is unramified at $p$).
\item If $\psi$ has order $2$, $\psi$ must be the (unique) quadratic
character ramified at $p$. Then $E$ is the twist of an unramified principal series, i.e., $E_p$ has good reduction at $p$.
\item If $\psi$ has order $3$, $4$ or $6$, there exists a newform
$g \in S_2(\Gamma_0(p \cdot m),\varepsilon)$, where $\varepsilon =
\psi^{-2}$, such that $f_E = g \otimes \psi$. In particular $\varepsilon$ has
 always order $2$ or $3$.
\end{itemize}
In the last case, the form has inner twists, since the Fourier
coefficients satisfy that $\overline{a_p} = a_p \varepsilon^{-1}(p)$
(see for example \cite[Proposition $3.2$]{RibetMF1V}).

\begin{rem}
  The newform $g$ can be taken to be the same for $E$ and
  $E_p$.
\label{rem:order3}
\end{rem}

\subsection{The case $\psi$ has order $2$.} This case is very similar
to the one treated  in the previous section. The curve $E_p$ has good reduction at $p$,
and is isomorphic via $\phi$ to $E$. It is quite easy to see that under these conditions $\sign(E,K)=\sign(E_p,K)=-1$.
Exactly as before we can construct Heegner points on $E_p$ and transfer them to $E$.

\subsection{The case $\psi$ has order $3$, $4$ or $6$.}\label{subsection:346} Let $d$ be
the order of $\psi$. 
Let ${g \in S_2(\Gamma_0(p \cdot m),\varepsilon)}$ as before. Suppose its
$q$-expansion at the infinity cusp is given by $g=\sum a_{n}q^{n}$.
Following \cite{Ribet}, we define the coefficient field
$K_{g}:=\QQ(\left\{a_n \right\})$.

\begin{rem}
  $K_g$ is an imaginary quadratic field
  generated by the values of $\psi$. It is equal to  $\QQ(i)$ if $d=4$ and 
  to $\QQ(\sqrt{-3})$ if $d=3$ or $d=6$.
\end{rem}

There is an abelian variety $A_{g}$ defined over $\QQ$ attached to $g$
via the Eichler-Shimura construction, with an action of $K_g$ on it,
i.e. there is an embedding $\theta:K_g \hookrightarrow
(\End_\QQ(A_g)\otimes \QQ)$. The variety $A_g$ can be defined as the
quotient $J_{1}(p \cdot m)/ I_{g} J_{1}(p \cdot m)$ where $I_{g}$ is the annihilator
of $g$ under the Hecke algebra acting on the Jacobian. Moreover, the
L-series of $A_g$ satisfies the relation
\[ 
L(A_g/\QQ,s)=L(g,s)L(\overline{g},s). 
\]
The variety $A_{g}$ has dimension $[K_{g}:\QQ]=2$ and is
$\QQ$-simple. However, it is not absolutely simple. The
variety $A_{g}$ is isogenous over $\overline{\QQ}$ to the square of
an elliptic curve (called a building block for
$A_{g}$, see \cite{GL} for the general theory). 

Under our hypotheses we have an explicit description. Let
$L= \overline{\QQ}^{\ker(\varepsilon)}$ (which is the splitting field
of $A_g$). It is a cubic extension if $d=3,6$ (and in particular
$p \equiv 1 \pmod{3}$) and the quadratic extension $\QQ(\sqrt{p})$ if
$d=4$ (which implies $p \equiv 1 \pmod 4$).  Let $M$ be the extension
$\overline{\QQ}^{\ker(\psi)}$.

\begin{prop}
\begin{itemize}
\item There exists an elliptic curve $\tilde{E}/L$ and an isogeny,
 defined over $L$, $\omega: A_{g} \rightarrow
  \tilde{E}^2$.  Furthermore, if $d=3$ $($resp. $d=6)$
  $\tilde{E}=E$ $($resp. $\tilde{E}=E_p)$ while if $d=4$,
  $\tilde{E}$ is the quadratic twist of
  $E/\QQ(\sqrt{p})$ by the unique quadratic extension unramified outside
  $p$.
\item In any case, there exists an isogeny $\varphi: A_{g} \rightarrow E^2$ defined over $M$.
\end{itemize}

\label{prop:splitting}
\end{prop}

\begin{proof}
  $A_g \simeq E^2$ over $M$ because (on the
  representation side) the twist becomes trivial while restricted to
  $M$, so the L-series of $A_g$ becomes the square of that of $E$ (over such
  field) and by Falting's isogeny Theorem there exists an isogeny
  (defined over $M$).
If $d=3$, $\varepsilon=\psi^2$ and $M=L$, while if $d=6$, starting
with $E_p$ (whose character has order $3$) gives the result. If $d=4$,
it is clear (on the representation side) that
$L(A_g,s) = L(\tilde{E},s)$ over $L$, where $\tilde{E}$ is the twist
of $E$ (while looked over
$\bar{\QQ}^{\ker(\varepsilon)} = \QQ(\sqrt{p})$) by the quadratic
character $\psi^2$. Then Falting's isogeny Theorem proves the claim.
\end{proof}

\begin{prop}
  Let $\sigma \in \Gal(\bar{\QQ}/\QQ)$. Then
  $\varphi^\sigma:A_g \to E^2$ is equal to
  $ \varphi \kappa(\left. \sigma \right|_M) $, where $\kappa$ is some
  character of $\Gal(M/\QQ)$ of order $[M:\QQ]$.
  \label{prop:conjugateisogeny}
\end{prop}
\begin{proof}
Since $\varphi$ and $\varphi^\sigma$ are isogenies of the same degree there exists an element
$a_\sigma \in  \End(A_g) \otimes  \QQ=K_{g}$ of norm $1$ such that $\varphi^{\sigma} = \varphi a_{\sigma}$. The map
$\kappa(\left. - \right|_M) : \Gal(\bar{\QQ}/\QQ) \rightarrow K^{\times}_{g}$, given by sending $\kappa(\left. \sigma \right|_M) \mapsto a_{\sigma}$
is a character, since the endomorphism $a_{\sigma}$ is defined over $\QQ$.
 Clearly $\kappa$ has the predicted order since otherwise the isogeny $\varphi$
  could be defined over a smaller extension (given by the fixed field
  of its kernel), which is not possible.
\end{proof}

In order to explicitly compute Heegner points it is crucial to have a better understanding of the isogenies $\omega$ and $\varphi$.
 Let us recall some basic properties of Atkin-Li operators
for modular forms with nebentypus, as explained in \cite{Atkin-Li}.
Let $N$ be a positive integer, and let $P \mid N$ be such that
$\gcd(P,N/P)=1$. Let $N'=\frac{N}{P}$  and decompose $\varepsilon =\varepsilon_P \varepsilon_{N'}$, where each
character is supported in the set of primes dividing the sub-index.

\begin{thm}
  Assuming the previous hypotheses, there exists an operator
  $W_P:S_2(\Gamma_0(N),\varepsilon) \to
  S_2(\Gamma_0(N),\overline{\varepsilon_P} \varepsilon_{N'})$ which satisfies the following properties:
  \begin{itemize}
  \item $W_P^2 = \varepsilon_P(-1) \overline{\varepsilon_{N'}}(P)$.
  \item If $g$ is an eigenvector for $T_q$ for some prime $q \nmid N$
    with eigenvalue $a_q$, then $W_p(g)$ is an eigenvector for $T_q$
    with eigenvalue $\overline{\varepsilon_P}(q) a_q$.
  \item If $g \in S_2(\Gamma_0(N),\varepsilon)$ is a newform, then
    there exists another newform $h \in
    S_2(\Gamma_0(N),\overline{\varepsilon_P}\varepsilon_{N})$ and a
    constant $\lambda_P(g)$ such that $W_P(g) = \lambda_P(g) h$.
  \item The number $\lambda_P(g)$ is an algebraic number of absolute
    value $1$. Furthermore, if $a_P$, the $P$-th Fourier coefficient
    of the newform $g$, is non-zero then
\[
\lambda_P(g)=G(\varepsilon_P)/a_P,
\]
where $G(\chi)$ denotes the Gauss sum of the character $\chi$.
  \end{itemize}
\label{thm:AtLi}
\end{thm}

The number $\lambda_P(g)$ is called the pseudo-eigenvalue of $W_P$ at $g$.
 
\begin{proof}
  See \cite[Propositions 1.1, 1.2, and Theorems 1.1 and 2.1]{Atkin-Li}.
\end{proof}
In our setting $N=p \cdot m$, $P=p$, $\varepsilon_{N'}$ is trivial, and $W_p$ is
an involution (i.e. $W_p^2=1$) acting on the differential forms of
$A_g$.

If $\eta$ is an endomorphism of $J(\Gamma_1(N))$ (or one of its
quotients), we denote $\eta^*$ the pullback it induces on the
differential forms. Given an integer $u$ let $\alpha_u$ be the
endomorphism of $J(\Gamma_1(N))$ corresponding to the action of the matrix 
$\left(\begin{smallmatrix} 1 & u/p\\ 0 & 1\end{smallmatrix}\right)$ on differential forms. Such endomorphism is defined over the cyclotomic field of
$p$-th roots of unity.

Let $\tau \in \Gal(K_g/\QQ)$ denote complex conjugation. Recall that
$\tau{a_q}=a_q \varepsilon^{-1}(q)$ for all positive integers $q$
prime to $p \cdot m$. Following \cite{Ribet2} we define
\[ 
\eta_{\tau} = \sum_{u \; (\text{mod }{p})}\varepsilon (u) \alpha_u.
\]
Since $\varepsilon(u) \in \On_{K_g}$, via the map $\theta$ we think of
$\eta_\tau$ as an element in $\End_{L}(A_g)$.  To normalize
$\eta_{\tau} $ we follow \cite{GL}. Let $a_p \in K_g$ be the $p$-th
Fourier coefficient of $g$.
\begin{lemma}
  The element $a_p$ has norm $p$.
\end{lemma}
\begin{proof}
  Looking at the curve $E$ over $\QQ_p$, the coefficient $a_p$ is one
  of the roots of the characteristic polynomial attached to the
  Frobenius element in the minimal (totally ramified) extension where
  $E$ acquires good reduction (see for example Section 3 of
  \cite{Tim}). Since the norm of the local uniformizer in such
  extension is $p$ (because the extension ramifies completely) the
  result follows.
\end{proof}
We then consider the normalized endomorphism $\frac{\eta_{\tau}}{a_p}$. 
\begin{rem}
  Our choice is a particular case of the one considered in \cite{GL},
  since our normalization corresponds to the splitting map
  $\beta:\Gal(K_g/\QQ)\to K_g^\times$ given by $\beta(\tau)=a_p$.
\end{rem}

\begin{thm} The operator $W_p$ coincides with
  $\left(\frac{\eta_{\tau}}{a_p}\right)^*$.
\label{thm:A-Lrelation}
\end{thm}

\begin{proof}
  It is enough to see how it acts on the basis $\{g,\overline{g}\}$ of
  differential forms of $A_g$. By Theorem~\ref{thm:AtLi} (since $a_p$
  is non-zero), $W_p(g) = \lambda_p \overline{g}$, where $\lambda_p =
  G(\varepsilon)/a_p$. On the other hand,
  $\eta_\tau(f)=G(\varepsilon)\overline{f}$, by
  \cite[Lemma $2.1$]{GL}.
  Exactly the same argument applies to $\overline{g}$, using the fact
  that $\overline{G(\varepsilon)}=G(\overline{\varepsilon})$, since
  $\varepsilon$ is an even character.
\end{proof}

\begin{coro}
  The Atkin-Li operator $W_p$ is defined over $L$, i.e. it corresponds
  to an element in $\End_{L}(A_g) \otimes \QQ$. Its action decomposes
  as a direct sum of two $1$-dimensional spaces.
\end{coro}
Let 
\[
\omega : A_g \to (W_p+1)A_g \times (W_p-1)A_g.
\]
Then both terms are $1$-dimensional, and the isogeny $\omega$ gives a
splitting as in Proposition~\ref{prop:splitting}.

\begin{rem}
  The explicit map $\omega$ satisfies the first statement of
  Proposition~\ref{prop:splitting}. In order to get the second
  statement we need to eventually compose it the isomorphism between
  $\tilde{E}$ and $E$. Recall that $E =\tilde{E}$ if $d=3$ and
  $\tilde{E}$ is a quadratic twist of it otherwise, so in any case the
  isomorphism is easily computed.
\end{rem}

\subsection{Heegner points}
This section follows Section $4$ of \cite{darmon2012birch}, so we
suggest the reader to look at it first.  Keeping the notations of the
previous sections, let $\varepsilon:(\ZZ/p)^\times\to \CC^\times$ be a
Dirichlet character. Extend the character to
$(\ZZ/{p \cdot m})^\times$ by composing with the canonical projection
$(\ZZ/{p \cdot m})^\times \rightarrow (\ZZ/p)^\times$ and define
\[
\Gamma^{\varepsilon}_{0}(p \cdot m) :=\left\lbrace
  \left(\begin{smallmatrix} a & b\\ c & d\end{smallmatrix}\right) \in \Gamma_0(p \cdot m) :
  \varepsilon(a)=1 \right\rbrace.
\]
Let $X^{\varepsilon}_{0}(p \cdot m)$ be the modular curve obtained as the
quotient of the extended upper half plane $\mathcal{H}^{*}$ by this
group. This modular curve has a model defined over $\QQ$ and it
coarsely represents the moduli problem of parameterizing quadruples
$(E,Q,C,[s])$ where
\begin{itemize}
\item $E$ is an elliptic curve over $\CC$,
\item $Q$ is a point of order $m$ on $E(\CC)$,
\item $C$ is a cyclic subgroup of $E(\CC)$ of order $p$,
\item $[s]$ is an orbit in $C \setminus \left\lbrace 0 \right\rbrace$ under the
  action of $\ker(\varepsilon) \subset (\ZZ/p)^{\times}$.
\end{itemize}

\begin{rem}
  There is a canonical map $\Phi: X^{\varepsilon}_{0}(p \cdot m) \to X_0(p \cdot m)$ which is
  the forgetful map in the moduli interpretation. This map has degree
  $\ord(\varepsilon)$.
\end{rem}
As in the classical case, there exists a modular parametrization 
\[
\xymatrix{
 X^{\varepsilon}_{0}(p \cdot m)\ar[d]_\Phi\ar[r]^-{\Psi} {\ar@/_1pc/[rr]_{\Psi_g}}& \Jac(X^{\varepsilon}_{0}(p \cdot m)) \ar[r]^-\pi& A_{g}\\
X_0(p \cdot m)\ar@/_/@{.>}[urr]
}
\]
where $\Psi(P)=\left( P \right) - \left( \infty \right)$ (the usual
immersion of the curve in its Jacobian) and $\pi$ is the
Eichler-Shimura projection onto $A_{g}$. These maps are defined over
$\QQ$, as the cusp $\infty$ is rational.
Our strategy is to  construct Heegner points on
$X^{\varepsilon}_{0}(p \cdot m)$ and push them through the modular
parametrization $\Psi_{g}$ to the abelian variety $A_{g}$ and finally
project them onto the elliptic curve $E$.
To construct points on $X^{\varepsilon}_{0}(p \cdot m)$, we consider the canonical map
\[
\Phi:X^{\varepsilon}_{0}(p \cdot m) \to X_0(p \cdot m),
\]
 and look at preimages of classical Heegner points on $X_0(p \cdot m)$.

Since the conductor  $p \cdot m$ satisfies the classical Heegner hypothesis with respect to
$K$ there is a cyclic ideal $\id{n}$ of norm $p \cdot m$.  Let $c$ be a
positive integer such that $gcd(c,p \cdot m)=1$. Then, a classical Heegner point
on $X_0(p \cdot m)$ corresponds to a triple
$P_{\id{a}}=(\Om_{c},\id{n},[\id{a}]) \in X_{0}(p \cdot m)(H_{c})$, where
$[\id{a}] \in \Pic(\Om_c)$.  Such point is represented by the elliptic
curve $E_{\id{a}}= \CC/ \id{a}$ and its $\id{n}$ torsion points
$E_{\id{a}}[\id{n}]$ (which are isomorphic to ($\id{a}{\id{n}}^{-1}/\id{a})$) are defined over $H_{c}$.  

The action of $\Gal (\overline{\QQ}/ H_{c})$ on $E_{\id{a}}[\id{n}]$
gives a map
$\Gal(\overline{\QQ}/H_c) \to
(\id{a}{\id{n}}^{-1}/\id{a})^\times$. Composing such map with the
character $\varepsilon$ gives
\[ 
\rho: \Gal(\overline{\QQ}/H_{c}) \rightarrow
(\id{a}{\id{n}}^{-1}/\id{a})^{\times} \stackrel{\varepsilon}{\rightarrow} \CC^{\times} .
\]
Its kernel corresponds to an extension $\tilde{H}_{c}$ of degree $\ord(\varepsilon)$ of $H_{c}$. Let $\tilde{H_c}=H_{c} M$.

\begin{prop}
  The $\ord(\varepsilon)$ points $\Phi^{-1}(P_{\id{a}})$ lie in
  $X^{\varepsilon}_{0}(p \cdot m)(\tilde{H_c})$ and are permuted under
  the action of $\Gal(\tilde{H_c}/H_c)$.
\end{prop}
\begin{proof}
  By complex multiplication $\tilde{H_c}$ lies in the composition of
  $H_{c}$ and the ray class field $K_{\id{p}}$, where $\id{p}$ is the
  unique prime of $K$ dividing $p$. The composition $H_{c}K_{\id{p}}$
  equals $H_{c}(\xi_{p})$, where $\xi_{p}$ is a
  $p$-th root of unity. Note that
  $\QQ(\sqrt{p^ {\ast}}) \subset H_{c}$ and the extension
  $H_{c}/K$ is unramified at $p$. Therefore, the unique extension of
  degree $ord(\varepsilon)$ of $H_{c}$ lying inside $H(\xi_{p})$ is
  given by $H_{c} \bar{\QQ}^ {ker(\psi)}=H_{c}M$.
\end{proof}
Using the aforementioned moduli interpretation, points on
$X^{\varepsilon}_{0}(p \cdot m)$ represent quadruples
$\left( \Om_{c}, \id{n}, \left[ \id{a} \right],[t] \right)$ where
$[t]$ is an orbit under $\ker(\varepsilon)$ inside
$(\Om_{c}/ (\id{n}/\id{p}))^{\times}$.
\begin{rem}  
  Let $\sigma \in  \Gal(\tilde{H_{c}}/K) $. Its action
 on Heegner points is given by
\[
 \sigma \cdot(\Om_{c},\id{n},[\id{a}] ,[t])=(\Om_{c}, \id{n}, [\id{a}{\id{b}}^{-1}],[dt]),
\]
where  $\sigma \mid_{H_c}=\text{Frob}_{\id{b}}$,
and $d=\rho(\sigma) \in {\Om_{c}/(\id{n}/\id{p})}^{\times}$. 
\label{rem:isogeny}
\end{rem}

\subsection{Zhang's formula}
\begin{thm}[Tian-Yuan-Zhang-Zhang]
  Let $K$ be an imaginary quadratic field satisfying the Heegner
  hypothesis for $p \cdot m$ and let
  $\tilde{\chi} : \mathbb{A}^{\times}_{K} \rightarrow \CC^{\times}$ be
  a finite order Hecke character such that
  $\tilde{\chi} \mid_{\mathbb{A}^{\times}_{\QQ}} = \varepsilon^{-1}$.
  Then $L(g,\tilde{\chi},s)$ vanishes at odd order at
  $s=1$. Moreover, if such order equals $1$,
  $(A_{g}(\tilde{H_{c}})) \otimes \CC))^{\tilde{\chi}}$ has rank one
  over $K_{g} \otimes \CC$. 

  More precisely, consider the Heegner point
  $\left( \left[ \mathfrak{a} \right],\mathfrak{n},1 \right) \in
  X^{\varepsilon}_{0}(p \cdot m)(\tilde{H_{c}} )$ and denote by
  $P_{c}$ its image under the modular parametrization $\Psi_{g}$. Then
\[ 
P^{\tilde{\chi}}= \sum_{\sigma \in \Gal(\tilde{H_{c}}/K)} \bar{\tilde{\chi}}(\sigma) P_{c}^{\sigma}  \in (A_{g}(\tilde{H_{c}})\otimes \CC)^{\tilde{\chi}}    
\]  
generates a rank one subgroup over $K_{g} \otimes \CC$.
\label{thm:chinos}
\end{thm}
\begin{proof}
  See \cite[Theorem 4.3.1]{tian2003euler}, \cite{zhang2010arithmetic}, and \cite[Theorem 1.4.1]{yuan2013gross}.
\end{proof}

Let $c$ be a positive integer relatively prime to
$\Disc(K) \cdot p \cdot m$, and let $\chi$ be a ring class character
of $\Gal(H_c/K)$. Since $\bar{\kappa}^ 2= \varepsilon^{\pm 1}$, the
character
$\tilde{\chi} : \Gal(\tilde{H_{c}}/K) \rightarrow \CC^{\times}$ given
by $\tilde{\chi}=\chi \bar{\kappa}$ satisfies the hypothesis of
Theorem~\ref{thm:chinos} (for either $g$ or its conjugate $\bar{g}$).
Summing up, we get the following theorem:

\begin{thm}
\label{thm:explicit}
 The point  $\varphi( P^{ \chi \bar{ \kappa} })$ belongs to $(E^2(H_{c} \otimes \CC))^{\chi}$. In addition, it is non-torsion if and only if  
$ L'(E/K, \chi ,1) \neq 0$.
\end{thm}
\begin{proof} By definition and Proposition~\ref{prop:conjugateisogeny}.
\[
\varphi( P^{ {\chi}\bar{ \kappa} })=\sum_{\sigma \in \Gal (\tilde{H_{c}}/K) }  \bar{ \chi}(\sigma)  \varphi( \kappa(\sigma)  P ^\sigma)=  \sum_{\sigma \in \Gal (\tilde{H_{c}}/K) }  \bar{ \chi}(\sigma)  {\varphi( P)}^{\sigma} , \]
so it lies in the right space.
  Since $\ord(\kappa)=\ord(\psi)$ and
  ${\bar{\kappa}}^2=\varepsilon^{\pm 1}$ we get
  $\bar{\kappa}=\psi^{\pm 1}$. We know that
  $g \otimes \psi=\bar{g} \otimes {\psi}^{-1}=f_E$, therefore,  using  $g$ or $\bar{g}$, we obtain $L(g,\tilde{\chi},s)=L(E,{\chi},s)$. 
  Theorem~\ref{thm:chinos} and the previous result imply that
  $\varphi( P^{ \chi \bar{ \kappa} }) \in (E^2(H_{c} \otimes
  \CC))^{\chi}$ is non-torsion if and only if $ L'(E/K, \chi ,1) \neq 0$.
 \end{proof}
Once we construct a non-torsion point on $E \times E$ we can project it to some coordinate in order to obtain a non-torsion point on $E$.

\subsection{Heegner systems}
As in the classical case, the family of Heegner points constructed
using different orders satisfy certain compatibilities.  

\begin{prop} Let $\ell$ be a prime such that $\ell \nmid N$ and $\ell$
  is inert in $K$. Then for every Heegner point
  $P_{c\ell} \in A_{g}(\tilde{H}_{c\ell})$ there exists a Heegner
  point $P_{c} \in A_{g}(\tilde{H_{c}})$ with
\begin{equation}
  \label{eq:kolyvagin}
\Tr_{ \tilde{H}_{c\ell}/\tilde{H_{c}}} P_{c \ell}= \theta(a_{\ell}) P_{c}  ,  
\end{equation}
where $a_{\ell}$ is the $\ell$-th Fourier coefficient of $g$.
\end{prop}
\begin{proof}
  The proof mimics the classical case one (see \cite[Proposition $3.7$]{Gr89}).
\end{proof}
To construct a point on $E$, we first apply the isogeny $\varphi$ to a
point in $A_g$ and then project onto one of the coordinates (call
$\pi_i$ the projection to the $i$-th coordinate). But $K_g$ does not
act on $E$! To overcome this problem, we restrict to primes $\ell$
which split completely in $L$. Let
$Q_{c}:= \pi_{i}(\Tr_{ \tilde{H_{c}}/H_{c}} \varphi (P_{c})) \in
E(H_{c})$.
\begin{prop}
Let $\ell$ be a prime such that $\ell \nmid N$, $\ell$ is inert in $K$ and $\ell$ splits completely in $L$. Then for every Heegner point $Q_{c\ell} \in E(H_{c\ell})$ there exists a Heegner
point $Q_{c} \in E(H_{c})$ such that 
\[ 
\Tr_{H_{c \ell}/H_{c}} Q_{c \ell} = a_{\ell} Q_{c} .
\]
\label{prop:Heegner-compatibility}
\end{prop}
\begin{proof}
  Applying $\pi_{i}(\Tr_{ \tilde{H_{c}}/H_{c}} \varphi)$ to equation
  (\ref{eq:kolyvagin}), since $\varphi$ commutes with the trace and 
  $a_{\ell} \in \QQ$ (because $\ell$ splits completely in
  $L$) we get
\[ 
\pi_{i}(\Tr_{ \tilde{H_{c}}/H_{c}}  \Tr_{ \tilde{H}_{c\ell}/\tilde{H_{c}}} \varphi(P_{c \ell}))= a_{\ell} Q_{c} .\]
Also 
\[ 
\pi_{i}(\Tr_{ \tilde{H_{c}}/H_{c}} \Tr_{ \tilde{H_{c\ell}}/\tilde{H_{c}}} \varphi(P_{c \ell}))=  \pi_{i}(\Tr_{ H_{c \ell}/H_{c}}  \Tr_{ \tilde{H_{c\ell}}/H_{c \ell}} \varphi(P_{c \ell})),
\]
but since $\pi_{i}$ is defined over $\QQ$, this expression equals $\Tr_{ H_{c \ell}/H_{c}} Q_{c \ell}$ as claimed.
\end{proof}

The previous results are enough for proving a Kolyvagin-type
theorem. 
\begin{thm}[Kolyvagin, Bertolini-Darmon] If $\pi_i(\varphi(P^{\chi \bar{\kappa}}))$ is non-torsion, then $\dim_{\CC}(E(H_c))^\chi=1$.  
\end{thm}

\begin{proof}
  The proof is very similar to the one given in
  \cite{Bertolini-Darmon} (Theorem 2.2) with the following remarks
  (using their notation and terminology): any $p$-descent prime is
  automatically unramified in $L$ hence $K(E[p])$ and $L$ are
  disjoint. We also require special rational primes $\ell$ to split
  completely in $L/\QQ$. Recall that $L$ is totally real, hence such
  condition is compatible with the other ones and special primes do
  exist. The first assertion of Proposition $3.2$ in
  \cite{Bertolini-Darmon} is exactly our
  Proposition~\ref{prop:Heegner-compatibility}, and the second one
  follows from \cite{Gr89} (proof of Proposition 3.7). With these
  modifications, the proof of \cite{Bertolini-Darmon} holds.
\end{proof}
\section{General case}

While considering the case of many primes ramifying in $K$, it
is clear that the potentially multiplicative case works similarly.
Some extra difficulties arise in the other cases. To make the
exposition/notation easier, we start considering the following two
cases:

\noindent {\bf Case 1:} Suppose that the conductor of $E$ equals
$p_1^2 \cdots p_r^2\cdot m$ where:
\begin{itemize}
\item $E$ has potentially good reduction at all $p_i$'s over an
  abelian extension,
\item all characters $\psi_{p_i}$ have the same order, 
\item all $p_{i}$'s are ramified in $K$,
\item $m$ satisfies the classical Heegner hypothesis.
\end{itemize}
Let $P=\prod_{i=1}^r p_i$. There are $2^r$ newforms of level
$P\cdot m$ which are twists of $f$ (obtained, following the previous
section notation, by twisting $f_E$ by all possible combinations of
$\{\psi_{p_i},\overline{\psi_{p_i}}\}$). Working with all of them
implies considering an abelian variety of dimension $2^r$, but the
coefficient field has degree $2$ so such variety is not simple over
$\QQ$.

Instead, take ``any'' newform
$g \in S_2(\Gamma_0(P\cdot m),\varepsilon)$, and consider the abelian
surface $A_g$ attached to it by Eichler-Shimura. The only Atkin-Li
operator acting on (the space of holomorphic differentials of) such
variety is the operator $W_P$, which again is an involution, so we can
split the space in the $\pm 1$ part and proceed as in the
previous case considered (where the splitting map is determined by
$\beta(\tau)=\prod_{i=1}^r a_{p_i}$).

The ambiguity on the choice of $g$ is due to the following: the
operators $W_{p_i}$ act transitively on the set of all newforms
$g$. In particular they ``permute'' the different abelian surfaces
(note that such operators are not involutions, but have eigenvalues in
the coefficient field $K_g$ which is independent of $g$). Although
surfaces attached to different choices of $g$ are in general not
isomorphic (the traces of the Galois representations are different),
they become isomorphic over $M$ hence all of them give the same
Heegner points construction.

\medskip

\noindent{\bf Case 2:} Suppose the conductor of $E$ equals $p^2 \cdot q^2 \cdot m$, where
\begin{itemize}
\item $E$ has potentially good reduction at $p$ and $q$ over an abelian extension,
\item the order of $\psi_p$ equals $4$ and that of $\psi_q$ equals $3$,
\item both $p$ and $q$ ramify in $K$,
\item $m$ satisfies the classical Heegner hypothesis.
\end{itemize}
With such assumptions the coefficient field $K_{g}$ equals
$\QQ(\sqrt{-1},\sqrt{-3})$.  Let
$g \in S_2(\Gamma_0(p q m),\varepsilon)$ be any twist of $f$, obtained
by choosing local characters $\psi_p$ at $p$ and $\psi_q$ at $q$ (so
$\varepsilon = \psi_p^2 \psi_q^2$).  By Eichler-Shimura there exists a
$4$ dimensional abelian variety $A_g$ defined over $\QQ$ (attached to
$g$) and an embedding $K_g \hookrightarrow \End(A_g)\otimes \QQ$. The
Atkin-Li operators $W_p$ and $W_q$ do act on the differential forms
of $A_g$ although not necessarily as involutions. Since their
eigenvalues lie in $K_g$, we can diagonalize them.

Let $\sigma_i$ denote the Galois automorphism of $K_g$ which fixes
$\sqrt{-3}$ and $\sigma_{\sqrt{-3}}$ be the one fixing $\sqrt{-1}$ (so
their composition is complex conjugation). We have the following
analogue of Theorem~\ref{thm:A-Lrelation}.
\begin{thm} With the previous notations:
  \begin{enumerate}
  \item the operator $W_p$ coincides with
    $\left(\frac{\eta_{\sigma_i}}{a_p}\right)^*$,
\item the operator $W_q$ coincides with
  $\left(\frac{\eta_{\sigma_{\sqrt{-3}}}}{a_q}\right)^*$,
\item the operator $W_{pq}$ coincides with
  $\left(\frac{\eta_{\sigma_i \sigma_{\sqrt{-3}}}}{a_p a_q}\right)^*$.
  \end{enumerate}
\end{thm}
\begin{proof}
  The proof mimics that of Theorem~\ref{thm:A-Lrelation}. Consider the
  basis of differential forms given by
  $\{g,\overline{g},h, \overline{h}\}$, where
  $h \in S_2(pqm,\overline{\varepsilon_p}{\varepsilon_q})$ equals $\sigma_i(g)$. By
  Theorem~\ref{thm:AtLi}: 
\[
  W_p\, g = \frac{G(\varepsilon_p)}{a_p} \, {h}, \qquad W_p\,
  \overline{g}= \frac{\overline{G({\varepsilon_p})}}{\overline{a_p}}
  \,\overline{h}, \qquad W_p\, h =
  \frac{\overline{G(\varepsilon_p)}}{\overline{a_p}} \, {g}, \qquad
  W_p\, \overline{h}= \frac{{G({\varepsilon_p})}}{{a_p}} \, \overline
  g.
\]
A splitting map is given by
\begin{align}
  \label{eq:splittinggeneralcase}
  \beta(\sigma_i)=a_p, & &\beta(\sigma_{\sqrt{-3}})=a_q,& & \beta(\sigma_i \sigma_{\sqrt{-3}})=a_pa_q \psi_{p}(q) \psi_{q}(p),
\end{align}
By \cite[Lemma 2.1]{GL} we have
\begin{align*}
\left(\frac{\eta_{\sigma_i}}{a_p}\right)^*g = \frac{G(\varepsilon_p)}{a_p} h, & &\left(\frac{\eta_{\sigma_i}}{a_p}\right)^*\overline g = \frac{\overline{G(\varepsilon_p)}}{\overline{a_p}} \overline h, & &  \left(\frac{\eta_{\sigma_i}}{a_p}\right)^*h = \frac{\overline{G(\varepsilon_p)}}{\overline{a_p}} g,\\
\left(\frac{\eta_{\sigma_i}}{a_p}\right)^* \overline{h} = \frac{{G(\varepsilon_p)}}{{a_p}} \overline g.
\end{align*}
The same computations proves the second statement, and the last one
follows from the fact that if $\chi,\chi'$ are two
characters of conductors $N$ and $N'$ with $(N:N')=1$, then
\begin{equation}
  \label{eq:Gaussrelation}
G(\chi \cdot \chi') = \chi(N')\chi'(N) G(\chi) G(\chi').  
\end{equation}
\end{proof}
Then we can split $A_g$ into four pieces over $M$ as in the previous
section.

\medskip

Although we considered only two particular cases, the general
construction follows easily from them. Just split the primes into
three sets: the ones with potentially multiplicative reduction, the
ones with potentially good reduction with characters of order $4$ and
the ones with potentially good reduction with characters of order $3$
or $6$. Treat each set as in Case $1$, and use Case $2$ to mix
them. Note that in any case the abelian surface $A_g$ has dimension
$1$, $2$ or $4$.
\section{Examples}
In this section we show some examples of our construction, which were done using \cite{PARI2}. The
potentially multiplicative case is straightforward since we only have
to find the corresponding quadratic twist and then construct classical
Heegner points.  The potentially good case is a little more involved.
We consider the following two cases:

\noindent $\bullet$ The case where $\ord(\psi_p)=2$
works exactly the same as the previous one, since we only have to find the
quadratic twist. 

\noindent $\bullet$ In the case $\ord(\psi_p)=3, 4$ or $6$ we start by
applying Dokchitser's algorithm \cite{Tim} (see also the appendix in
\cite{Kohen}) to find $\psi_p$ as well as the corresponding Fourier
coefficient $a_{p}$ (which give the $q$-expansion of $g$). We compute
$A_g$ using the Abel-Jacobi map, and then we split it following
Section \ref{subsection:346}.

Each factor is isomorphic to $E$ over $M$.  To find the isomorphism
explicitly, we compare the lattices of $E$ and the one computed and
find one $\alpha \in M$ sending one lattice to the other. 
\begin{table}[h]
\scalebox{0.8}{
\begin{tabular}{||r|c||c|c||c||r|c||c|c||}
  \hline
N& E & St & Ps& $K$ & $\ord(\psi_p)$& $a_p$  & $\tau$ & $P$\\
  \hline
$5^2\cdot 29$ & \href{http://www.lmfdb.org/EllipticCurve/Q/725.a1}{.a1}& $\{5,29\}$ & $\emptyset$& $\QQ(\sqrt{-5})$ & &  & $\frac{45+\sqrt{-45}}{145}$ & $[8,8]$\\
\hline
$5^2\cdot 23$ & \href{http://www.lmfdb.org/EllipticCurve/Q/575.e1}{.e1} & $\{23\}$ & $\{5\}$ & $\QQ(\sqrt{-5})$ &$4$ & $2-i$  & $\frac{15+\sqrt{-5}}{5 \cdot 23}$ & $[\frac{-1637}{2^6}, \frac{-2^8 - 3 \cdot 5^2 \cdot 127
    \sqrt{-5} }{ 2^9} ]$\\
\hline
$2^2\cdot 7^2$ & \href{http://www.lmfdb.org/EllipticCurve/Q/196.b2}{.b2} & $\emptyset$ & $\{7\}$ & $\QQ(\sqrt{-7})$ & $3$ & $\frac{-5+\sqrt{-3}}{2}$ &  $\frac{21+\sqrt{-7}}{7 \cdot 2^3}$ & $[\frac{-139}{4},\frac{581\sqrt{-7}}{8}]$\\
\hline
$2\cdot3^2\cdot7^2$ & \href{http://www.lmfdb.org/EllipticCurve/Q/882.a1}{.a1} & $\{2\}$ & $\{7\}$ & $\QQ(\sqrt{-7})$ & $3$ & $\frac{-1+3\sqrt{-3}}{2}$ &  $\frac{21+\sqrt{-7}}{28}$ & $[39,15]$\\
\hline
$5^2\cdot 7^2$ & \href{http://www.lmfdb.org/EllipticCurve/Q/1225.d2}{.d2} & $\emptyset$ & $\{5\}$ &$\QQ(\sqrt{-35})$ & $4$ & $1-2i$&  $\frac{-35+\sqrt{-35}}{70}$ & $[-15,\frac{15+175\sqrt{-35}}{2}]$\\
               &    &             & $\{7\}$ & & $3$ & $\frac{1-3\sqrt{-3}}{2}$& & \\
\hline
\end{tabular}}
\caption{Examples of ramified primes}
\label{table:curvesdata}
\end{table}

The computations are summarized in Table~\ref{table:curvesdata}. The
table is organized as follows: the first two columns contain the curve
conductor and its label (following \cite{lmfdb} notation). The next two
columns list the principal series and the Steinberg primes of the
curve (following \cite{Pacetti} algorithm). The fifth column contains
the imaginary quadratic field. For the computations we just considered
the whole ring of integers. The sixth and seventh columns contain the
order of the character and the number $a_p$ for the principal series
primes ramifying in $K$. Finally the last two columns show the Heegner
points considered in the upper-half plane and the point constructed in
$E(K)$.

Some remarks regarding the examples considered:
\begin{itemize}
\item The first example corresponds to a potentially multiplicative
  case. The class number of $\mathcal{O}_K$ is $2$ and
  $H=\QQ(\sqrt{5},i)$. If $\chi_5$ denotes the non-trivial character
  of the class group, we can trace with respect to it and get the point
  $[9, \frac{-9+15 \sqrt{5}}{2}] \in E(H)^{\chi_{5}}$.
\item The second and third examples correspond to elliptic curves with
  only one potentially good reduction prime ramifying in $K$. The
  former has $ord(\varepsilon)=2$ while the latter has
  $ord(\varepsilon)=3$
\item The fourth example is quite interesting, since the prime $2$
  splits in $K$ (so we use an Eichler order at $2$), the prime $3$ is
  inert in $K$ (so we use a Cartan order at $3$), and the prime $7$ is
  ramified in $K$. This is a mixed case of the Cartan-Heegner
  hypothesis (as in \cite{Kohen}) and the present one.  We compute the
  $q$-expansion of $g$ (as explained in the aforementioned article) as
  a form in $S_{2}(\Gamma_{0}(2\cdot 7^2) \cap \Gamma_{ns}(3))$ and
  then twist by the character $\psi_7$ (of order $3$) to get a form in
  $S_{2}(\Gamma^{\varepsilon}_{0}(2\cdot 7) \cap \Gamma_{ns}(3))$. The
  results of Section~\ref{subsection:346} apply to give the
  corresponding splitting.
\item The last example corresponds to an elliptic curve with two
  primes of potentially good reduction ramifying in $K$, hence the
  coefficient field is $K_{g}=\QQ(\sqrt{-1},\sqrt{-3})$.
\end{itemize}

\appendix
\section{Computation of a Darmon point (by Marc Masdeu)}

Let $E$ denote the elliptic curve
\cite[\href{http://www.lmfdb.org/EllipticCurve/Q/147.c2}{147.c2}]{lmfdb},
of conductor $3\cdot 7^2$ which has potentially good reduction over
an abelian extension at the prime $7$. Let $K = \QQ(\sqrt{35})$, which
has class number $2$. The prime $3$ is inert in $K$, while $7$
ramifies.  It is easy to see that $\operatorname{sign}(E,K)=-1$.

Let $p = 3$ and consider the Dirichlet character $\chi$ of conductor
$7$ which maps $3\in (\ZZ/7\ZZ)^\times$ to $\zeta_6=e^{\pi i/3}$.  Let
$\Gamma$ denote the group
\[
\Gamma = \Gamma_0^\chi(7)[1/3] = \Big\{\smtx abcd\in \SL_2\left(\ZZ\left[1/3\right]\right)~|~ c \in 7\ZZ[1/3],\ \chi(a) = 1\Big\}.
\]

In the page \url{http://github.com/mmasdeu/} there is code available
to make computations with such groups. 

There is a $2$-dimensional irreducible component in the plus-part of
$H^1(\Gamma_0^\chi(21),\ZZ)$, which corresponds to the abelian surface
$A_g$. Let $\{g_1, g_2\}$ be an integral basis of this subspace,
normalized such that its basis vectors are not multiples of other
integral vectors. Following the constructions of~\cite{MR3384519} with
the non-standard arithmetic groups, each of these vectors yield a
cohomology class
\[
\varphi^{(i)}_E\in H^1(\Gamma, \Omega^1_{\cH_3}),\quad i =1,2.
\]
Here $\cH_3$ denotes the $3$-adic upper half-plane and $\Omega^1_{\cH_3}$ is the module of rigid-analytic differentials with $3$-adically bounded residues.

The ring of integers $\cO_K$ of $K$ embeds into $M_2(\ZZ)$ via
\[
\sqrt{35}\mapsto \psi(\sqrt{35}) = \left(\begin{array}{rr}
15 & 10 \\
-19 & -15
\end{array}\right).
\]

The fundamental unit of $K$ is $u_K=\sqrt{5}+6$, which is mapped to the matrix
\[
\psi(u_K) =
\left(\begin{array}{rr}
21 & 10 \\
-19 & -9
\end{array}\right).
\]

In order to obtain an element of $\Gamma_0^\chi(7)$ we need to consider $u_K^{14}$, which maps to
\[
\gamma_K = \psi(u_K)^{14} = \left(\begin{array}{rr}
-3057309462214237 & -4524404717310744 \\
2852342104391556 & 4221080735198699
\end{array}\right)\in \Gamma_0^\chi(7).
\]

The matrix $\gamma_K$ fixes a point $\tau_K$ in $\cH_3$, 
\[
\tau_K = 680113883076491926203393 + 188920523076803312834276\,\alpha_3 + O(3^{50}),
\]
where  $\alpha_3$ denotes a square root of $35$ in $K_3$, the completion of $K$ at $3$.

We present the above groups using Farey symbols so as to solve the
word problem for them. Although the homology class of
$\gamma_K\tns \tau_K$ might not lie in
$H_1(\Gamma_0^\chi(7),\Div^0\cH_3)$, its projection into the $A_g$
isotypical component is. It can be seen that such projection is given
by the operator $(T_2^2-3T_2+3)(T_2+3)$, where $T_2$ is the $2$-th
Hecke operator (just by computing the characteristic polynomial of the
Hecke operator $T_2$ in the whole space and computing its irreducible
factors). This allows to represent
$(T_2^2 - 3T_2 + 3)(T_2+3)(\gamma_K\tns \tau_K)$ 
by a cycle of the form
\[
\smtx{-6}{1}{-7}{1}\tns D_1 + \smtx{15}{-4}{49}{-13}\tns D_2 + \smtx{1}{1}{0}{1}\tns D_3 + \smtx{22}{-9}{49}{-20}\tns D_4 + \smtx{-13}{5}{-21}{8}\tns D_5,
\]
where $D_i$ are divisors of degree $0$ obtained by the aforementioned
code (each divisor has support consisting of more than a thousand
points in $\cH_3$).

This class was integrated against the cohomology classes
$\varphi^{(i)}_E$ using an overconvergent lift as explained
in~\cite{MR3384519} giving a point in $A_g(\CC_3)$ which can be
projected onto $E(\CC_3)$ by choosing an appropriate linear combination
of the basis elements. In the generic case any projection would
work. We have taken in this case the projection onto
$g_1$. Concretely, the integral corresponding to $\varphi^{(1)}_E$
resulted in the $3$-adic element
\[
J = 2 + (\alpha_3  + 2)\cdot 3 + 3^2 + (2\cdot \alpha_3  + 1)\cdot 3^3 + (\alpha_3  + 1)\cdot 3^5 + (\alpha_3  + 2)\cdot 3^6 + (\alpha_3  + 1)\cdot 3^7 +  \cdots + O(3^{120})
\]
If we apply Tate's uniformization (at $3$) to such point, we obtain a
point in $E(K_3)$ which coincides up to the working precision of
$3^{120}$ with
\[
14\cdot 13\cdot P = 14\cdot 13\cdot \left(\frac{164850\sqrt{7}}{2809} + \frac{610894}{2809}, \frac{63872781\sqrt{35}\sqrt{7}}{297754} + \frac{96772060\sqrt{35}}{148877} - \frac{1}{2} \right).
\]
Note that $P \in E(H)$, where $H=K(\sqrt{7})=\QQ(\sqrt{35},\sqrt{7})$
is the Hilbert class field of $K$ as would be predicted by the
conjectures. The factor $14$ appears because we took the $14$th power
of the fundamental unit, while the factor $13$ is due to the fact that
the point would naturally lie in the elliptic curve
\href{http://www.lmfdb.org/EllipticCurve/Q/147.c1}{147.c1}, which is
$13$-isogenous to $E$.

Finally, if one takes the trace of $P$ to $K$ one obtains:
\[
P_K = P + P^\sigma = \left(\frac{63367}{2000} , \frac{5823153}{200000}\sqrt{35} - \frac 12\right),\quad \Gal(H/K)=\langle \sigma\rangle,
\]
and one can check that $P_K$ is non-torsion and thus generates a subgroup of finite index in $E(K)$.

\bibliographystyle{alpha}
\bibliography{biblio}
\end{document}